\documentclass[hidelinks]{article}
\usepackage[utf8]{inputenc}
\usepackage{titlesec}
\usepackage[includehead,a4paper,headheight=13.6pt]{geometry} 
\usepackage[indent=15pt]{parskip} 
\usepackage{amssymb,amsmath,amsthm}
\usepackage{indentfirst}
\usepackage{hyperref}
\usepackage[nameinlink]{cleveref}

\newcommand\nnfootnote[1]{%
  \begin{NoHyper}
  \renewcommand\thefootnote{}\footnote{#1}%
  \addtocounter{footnote}{-1}%
  \end{NoHyper}
}

\newcommand{\address}{{
\bigskip
\footnotesize

S.F.~MacDonald, PhD student, \textsc{Department of Mathematics, University of Toronto, ON M5S2E4}\par\nopagebreak
\textit{E-mail address}: \texttt{sullivan@math.toronto.edu}
}}

\newtheorem{thm}{Theorem}
\newtheorem{lem}[thm]{Lemma}
\newtheorem{cor}{Corollary}[thm]

\newcommand{\re}{\mathrm{Re}\,}
\newcommand{\im}{\mathrm{Im}\,}
\newcommand{\rmk}{\noindent\textit{Remark}: }

\titleformat*{\section}{\large\bfseries}
\titleformat*{\subsection}{\normalsize\bfseries}

\title{Structural aspects of extremal functions in the Krzy\.z conjecture}

\author{Sullivan Francis MacDonald\footnote{University of Toronto, Canada}}
\date{}

\begin{document}
\maketitle

\nnfootnote{\textit{Date}: May 8, 2026}
\nnfootnote{\textit{2020 Mathematics Subject Classification}. 30C45, 30H05, 30J15.}
\nnfootnote{\textit{Key words and phrases}. Bounded analytic functions, Taylor coefficients, Non-linear extremal problems.}

\begin{abstract}
Extremal functions for the \(n\)th coefficient in the Krzy\.z conjecture are atomic singular inner functions with at most \(n\) atoms. This paper gives a lower bound on the number of atoms \(N\) of the form \(N\geq cn\), marking progress toward proving the expected \(N=n\). Furthermore, we prove new formulas for extremal functions using variational techniques.

Using these results and several other methods, we establish new conditions on extremal functions which are equivalent to the Krzy\.z conjecture being true. We also characterize the possible analytic invariants of extremal functions.
\end{abstract}

\section{Introduction and main results}

The Krzy\.z conjecture, posed by J. Krzy\.z \cite{Krzyz1968} in 1968, is a coefficient problem for functions in the class \(\mathcal B_0=\{f\in\textrm{Hol}(\mathbb D)\;|\; 0<|f|\leq1\}\). It states that if \(\smash{f(z)=\sum_{j=0}^\infty a_jz^j\in\mathcal B_0}\) and \(n\in\mathbb N\) then \(|a_n|\leq2/e\), and equality holds if and only if \(f(z)=\xi_1f_n(\xi_2 z)\), where \(\xi_1,\xi_2\in\partial\mathbb D\) and 
\begin{equation}\label{fn}
    f_n(z)=\exp\bigg(\frac{z^n-1}{z^n+1}\bigg).    
\end{equation}
Equivalently, defining \(M_n(f)=\re a_n\), the Krzy\.z conjecture is equivalent to the statement that if \(f\in\mathcal{B}_0\) and \(a_0>0\) then \(M_n(f)\leq 2/e\) for any \(n\in\mathbb N\), with equality only if \(f= f_n\). 

A thorough survey of progress on the problem up to 2015 is given by Mart\'in, Sawyer, Uriarte-Tuero \& Vukoti\'c \cite{MartinSawyerTueroVukotic2015}. The authors of \cite{MartinSawyerTueroVukotic2015} also collected many statements which they proved to be equivalent to the Krzy\.z conjecture. Many of the valuable ideas developed in \cite{MartinSawyerTueroVukotic2015} have opened new and promising directions for research, and several of our results build on their work directly. Agler \& McCarthy \cite{AglerMcCarthy2021} also showed that the conjecture would follow from a conjectured bound on an entropy functional for polynomials, along with a degree condition on extremal functions.

In a recent preprint, Lei \& Zhang \cite{LeiZhang2026} claim to prove Agler \& McCarthy's entropy conjecture. Such a development would mean that the Krzy\.z conjecture would follow if one could show that extremal functions have full degree. Our first main result, while not strong enough to prove this condition, gives the first non-trivial lower bound on this degree (a parameter which we denote by \(N\) below). Furthermore, we give conditional lower bounds on \(N\) that may be useful.

There have been several claims of proof of the Krzy\.z conjecture since 2009 by Krushkal \cite{Krushkal2009,Krushkal2016,Krushkal2020,Krushkal2022} using advanced properties of Teichm\"uller spaces. These works contain valuable insights, though the community does not seem to have reached a consensus about their completeness, and a significant level of expertise in complex geometry would be needed to fairly assess Krushkal's purported proofs. In any case, there remain aspects of the Krzy\.z conjecture to explore, and a proof using more accessible techniques would be valuable. 

The main results of this work concern extremal functions. A function \(f\in\mathcal B_0\) is called extremal for \(M_n\) if \(M_n(f)\geq M_n(F)\) for every \(F\in\mathcal{B}_0\). Such functions exist by a standard normal families argument. Indeed, Hummel et al. \cite{HummelScheinbergZalcman1977} showed that if \(n\in\mathbb N\) and \(f\) is extremal for \(M_n\), then there exists \(N\leq n\), \(\lambda_1,\dots,\lambda_N>0\), and distinct \(\alpha_1,\dots,\alpha_N\in\partial\mathbb D\) such that
\begin{equation}\label{extremal}
    f(z)=\exp\bigg(-\sum_{j=1}^N\lambda_j\frac{1+\alpha_jz}{1-\alpha_jz}\bigg).    
\end{equation}
This is called an atomic singular inner function, since the Herglotz measure of \(-\log f\) is a discrete atomic measure on the circle with point masses at \(\alpha_1,\dots,\alpha_N\) of respective weights \(\lambda_1,\dots,\lambda_N\).

If the Krzy\.z conjecture is true, then actually \(N=n\), \(\lambda_1=\cdots=\lambda_n=\frac{1}{n}\), and \(\alpha_1,\dots,\alpha_n\) are the \(n\)th roots of \(-1\). One can show that \(N\geq 2\) for \(n\geq 2\), but otherwise no lower bound on \(N\) has appeared in the literature. Our first result gives a modest lower bound that scales with \(n\).

\medskip

\begin{thm}\label{thm:1}
Fix \(n\in\mathbb N\) and let \(f\) be an extremal function for \(M_n\) as in \eqref{extremal}. There exists a constant \(c>0\), independent of \(n\) and \(N\), such that \(N\geq cn\). 
\end{thm}

\noindent\textit{Remark}: Our proof uses oscillatory integral methods similar to those of Horowitz \cite{Horowitz1978} and Ermers \cite{Ermers1990} to bound \(|a_n|\). We get \(c=\frac{2\pi}{7}(\frac{1}{e}-\frac{1}{3})\approx 0.03\), but this can likely be improved.

Many authors have used variational methods to extract coefficient identities for extremal functions. Using both new and existing identities, we develop formulas that encode this variational data. In what follows, we let \(T_j(F)\) denote the \(j\)th Taylor coefficient at zero of \(F\in\mathrm{Hol}(\mathbb D)\).

\medskip

\begin{thm}\label{thm:X}
Let \(f(z)=e^{g(z)}=\sum_{j=0}^\infty a_jz^j\) be extremal for \(M_n\), let \(g(z)=\sum_{j=0}^\infty b_jz^j\), and set
\[
    P(z)=a_n+2\sum_{j=1}^na_{n-j}z^j    \qquad\textrm{and}\qquad Q(z)=T_n(fg)+2\sum_{j=1}^n T_{n-j}(fg)z^j. 
\]
Then for all \(z\in \mathbb D\), it holds that
\[
    g(z)=\frac{z^nQ(1/z)-z^n\overline{Q(\overline{z})}}{z^nP(1/z)+z^n\overline{P(\overline{z})}}\qquad\textrm{and}\qquad zg'(z)=n-\frac{z^{n-1}P'(1/z)+z^{n+1}\overline{P'(\overline{z})}}{z^nP(1/z)+z^n\overline{P(\overline{z})}}.
\]
\end{thm}

\noindent\textit{Remark}: The polynomial \(P\) was studied by Mart\'in et al. \cite{MartinSawyerTueroVukotic2015}, who showed that \(\re P\geq 0\) in \(\partial \mathbb D\) and \(\re P(\alpha_j)=0\) for \(1\leq j\leq N\).

Any extremal function \(f\) is an inner function, so it has a unimodular non-tangential limit \(f(e^{i\theta})\) for a.e. \(\theta\in \mathbb T\). Thus, there exists \(\varphi:\mathbb T\to\mathbb R\) such that \(f(e^{i\theta})\overset{a.e.}{=}e^{i\varphi(\theta)}\). The preceding result gives us simple formulas for this function involving the coefficients of \(f\).

\medskip

\begin{cor}\label{cor:formulas for varphi}
If \(f\) is extremal for \(M_n\) and \(f(e^{i\theta})\overset{a.e.}{=}e^{i\varphi(\theta)}\), then 
\[
    \varphi(\theta)=\frac{\im Q(e^{-i\theta})}{\re P(e^{-i\theta})},\qquad \textrm{and}\qquad \varphi'(\theta)=n-\frac{\re e^{-i\theta}P'(e^{-i\theta})}{\re P(e^{-i\theta})}.
\]
\end{cor}

\rmk \textit{A priori} \(\varphi\) is only determined up to an integer multiple of \(2\pi \). For consistency, we always normalize so that \(f(0)\in\mathbb R\) and define \(\varphi\) using the radial limit of the branch of \(\log f\) whose value at zero is real.

Combining the second formula above with other properties of \(\varphi\), we obtain a lower bound on \(N\) in terms of the polynomial \(P\). 

\medskip

\begin{cor}\label{cor: lower bound on N}
Let \(f\) be extremal for \(M_n\), and let \(P\) be as above. If \(zP'(z)\) has \(m\) zeros in \(\mathbb D\) (counting multiplicities), then \(N\geq m\). If \(P'\) only vanishes in \(\mathbb D\), then \(N=n\). 
\end{cor}

Many criteria are known for a polynomial to have all its zeros in \(\mathbb D\). As \(P\) is determined by the coefficients of the extremal function \(f\), these translate to conditions on said coefficients. If \(f=f_n\) then \(P'(z)\) only vanishes at zero.

The function \(f_n\) in \eqref{fn} is invariant under the rotations \(r_k(z)=e^{2k\pi i/n}z\), in the sense that \(f_n\circ r_k\equiv f_n\) for \(1\leq k\leq n\). For \(n\geq 2\), these rotations are the only holomorphic invariants of \(f_n\). Thus, it is interesting then to ask what functions \(\psi\in\textrm{Hol}(\mathbb D)\) may act as invariants of a generic extremal function. 

\medskip

\begin{thm}\label{thm:4}
Fix \(n\geq 3\) and let \(f\) be extremal for \(M_n\). If \(\psi\in\mathrm{Hol}(\mathbb D)\) satisfies \(f\circ\psi=f\) in \(\mathbb D\), then  \(\psi\in\mathrm{Aut}(\mathbb D)\). If \(N\geq 3\) then there exists \(a\in\mathbb D\) and an \(N\)th root of unity \(\xi\) such that
\[
    \psi(z)=\psi_{-a}(\xi\psi_{a}(z)),\qquad \psi_a(z)=\frac{z-a}{1-\overline{a}z}.
\]
In particular, \(\psi(a)=a\), and \(\psi^{\circ N}=\mathrm{id}\). If \(N=2\) then either \(\psi\) is of the form above, or \(\psi\) fixes \(\{\overline{\alpha_1},\overline{\alpha_2}\}\) and no other points in \(\overline{\mathbb D}\).
\end{thm}

\noindent\textit{Remark}: The condition \(N\geq 3\) is automatically satisfied for \(n\geq 100\), thanks to Theorem \ref{thm:1}. 

If an extremal function \(f\) has a holomorphic invariant \(\psi\), then it takes on structural properties analogous to those for Blaschke products investigated by Cassier \& Chalendar \cite{CassierChalendar2000}, whose results we use in the proof of Theorem \ref{thm:4}. For instance, \(\psi\) acts as a permutation of \(\{\overline{\alpha_1},\dots,\overline{\alpha_N}\}\) and it fixes a point \(a\in\mathbb D\) when \(N\geq 3\). If \(a=0\) then we can say much more. 

\medskip

\begin{cor}\label{cor:4.1}
Fix \(n\in\mathbb N\) and let \(f\) be extremal for \(M_n\). If there exists \(\psi\in\mathrm{Hol}(\mathbb D)\setminus\{\mathrm{id}\}\) such that \(\psi(0)=0\) and \(f\circ\psi= f\), then \(\psi\) is a rotation and  \(\gcd(N,n)>1\).  
\end{cor}

Thus, any rotational invariant of an extremal function \(f\) would imply that \(N\) is at least as large as the smallest prime divisor of \(n\). If \(\psi^{\circ m}\neq\mathrm{id} \) for any divisor \(m>1\) of \(n\), then necessarily \(f\equiv f_n\) by \cite[Thm. 1(p)]{MartinSawyerTueroVukotic2015}. We do not show the existence of such an invariant.

The main result of \cite{MartinSawyerTueroVukotic2015} is a collection of conditions on extremal functions which are equivalent to the Krzy\.z conjecture. Here we give several new equivalent conditions.


\medskip

\begin{thm}\label{thm:3}
The Krzy\.z conjecture is true if and only if, whenever \(n\geq 3\) and \(f(z)=\sum_{j=0}^\infty a_jz^j\) is extremal for \(M_n\), any of the following are true:
\begin{itemize}
    \item[(1)] \(f(e^{i\theta}z)= f(z)\) in \(\mathbb D\) for some \(\theta\in (0,\frac{4\pi}{n})\).
    \item[(2)] \(\log f(z)=\log f(0)\) if and only if \(z=0\).
    \item[(3)] \(f'\) is non-vanishing in \(\mathbb D\setminus\{0\}\).
    \item[(4)] \(a_j=0\) for \(1\leq j\leq N-1\).
    \item[(5)] \(a_j=0\) for \(1\leq j\leq \lceil\frac{n-2}{3}\rceil\).
\end{itemize}
\end{thm}



\medskip

Finally, we show that for the polynomial \(P\) in Theorem \ref{thm:X}, if \(\re P\) has \(n\) zeros on \(\partial \mathbb D\) (as is true when \(f=f_n\)) then relatively simple formulas hold for the coefficients of \(f\). Recall that \(\re P\) vanishes at \(\alpha_1,\dots,\alpha_N\in\partial \mathbb D\), so what follows is valid when \(N=n\) but it is also true in a more general setting. Similar formulas were identified by Agler \& McCarthy \cite[Prop. 6.4]{AglerMcCarthy2021} in the special case when \(N=n\), though our techniques for proving them differ considerably.

\medskip

\begin{thm}\label{thm:2}
Fix \(n\in\mathbb N\) and let \(f(z)=\sum_{j=0}^\infty a_jz^j\) be an extremal function for \(M_n\) as in \eqref{extremal}, normalized so that \(a_0,a_n>0\). If \(\re P\) vanishes at distinct points \(z_1,\dots, z_n\in\partial \mathbb D\), then 
\[
    f(z)=a_0\prod_{j=1}^n(1-z_jz)^2\mod z^{n+1}.
\]
In particular, \(\prod_{j=1}^nz_j=(-1)^n\) and if \(e_1,\dots, e_n\) denote the elementary symmetric polynomials in \(n\) variables, then
\[
    a_n=a_0\bigg(2+\sum_{k=1}^{n-1}|e_k(z_1,\dots,z_n)|^2\bigg).
\]
\end{thm}

\rmk As mentioned above, if \(\re P(z_j)=0\) for \(1\leq j\leq n\) then  \(\{\alpha_1,\dots,\alpha_N\}\subseteq\{z_1,\dots, z_n\}\).

Under these conditions, the coefficients of \(\log f\) take on especially simple forms.

\medskip

\begin{cor}\label{cor:wow!}
Let \(f\) be extremal for \(M_n\) and write \(\log f(z)=\sum_{j=0}^\infty b_jz^j\). If \(\re P\) vanishes at distinct points \(z_1,\dots, z_n\in\partial \mathbb  D\), then 
\[
    b_k=-\frac{2}{k}\sum_{j=1}^nz_j^k,\qquad 1\leq k\leq n.
\]
\end{cor}

In conjunction with the formulas for \(b_1,\dots, b_n\) obtained directly from \eqref{extremal}, the identities above define a linear system for \(\lambda_1,\dots,\lambda_n\) that can be solved to write each \(\lambda_j\) as a function of \(z_1,\dots,z_n\). 

In general, one can prove that \(|b_0|\leq 2n\) and \(|b_j|\leq 2|b_0|\leq 4n\) for \(j\in\mathbb N \) (see Lemma \ref{lower bound a_0}). Corollary \ref{cor:wow!} shows that these bounds can be improved under the preceding hypotheses. 

\begin{cor}
Let \(f\) be extremal for \(M_n\), and write \(\log f=\sum_{j=0}^\infty b_jz^j\). If \(\re P\) vanishes at \(n\) distinct points on \(\partial \mathbb D\), then \(|b_j|\leq 2n/j\) for \(1\leq j\leq n\). In particular, \(|b_n|\leq 2\) and equality holds if and only if the zeros of \(\re P\) are equidistributed.
\end{cor}

Each of the theorems above is proved using different (though often related) techniques. So, the next five sections each begin with a brief overview of relevant preliminary material, followed by a proof of the corresponding main result. At the end of the paper, we collect assorted results that play no role in our major proofs but which may be of interest to the community.



The author gratefully acknowledges the support of an NSERC graduate research scholarship, as well as funding from the University of Toronto. The author also sincerely thanks his PhD supervisor, Prof. Ignacio Uriarte-Tuero, for providing consistent encouragement, valuable advice, and engaging conversation throughout the preparation of this paper.

\section{Oscillatory radial limits and the proof of Theorem \ref{thm:1}}

Fix \(n\in\mathbb N\) and let \(f(z)=\sum_{j=0}^\infty a_jz^j\) be an extremal function for \(M_n\). For \(\alpha_1,\dots,\alpha_N\) as in \eqref{extremal} write \(\alpha_j=e^{-i\theta_j}\). Direct computations show that the radial limit of \(\log f\) exists for all arguments in \([0,2\pi)\setminus\{\theta_1,\dots,\theta_N\}\) and it is given by
\begin{equation}\label{eq:radial_limit}
    \log f(e^{i\theta}):=\lim_{r\to 1^-}\log f(re^{i\theta})=-\sum_{j=1}^N\lambda_j\frac{1+e^{i(\theta-\theta_j)}}{1-e^{i(\theta-\theta_j)}}=-i\sum_{j=1}^N\lambda_j\cot\bigg(\frac{\theta-\theta_j}{2}\bigg).
\end{equation}
In fact, some more work shows that \(f(re^{i\theta_j})\to0\) as \(r\to 1^-\) for \(1\leq j\leq N\), so the radial limit of \(f\) exists everywhere. Taking \(\varphi(\theta)=-i\log f(e^{i\theta})\), away from \(\theta_1,\dots,\theta_N\) we have \(f(e^{i\theta})=e^{i\varphi(\theta)}\). Using Cauchy's integral formula, we can recognize the \(n\)th coefficient of \(f\) as a Fourier coefficient of this radial data:
\[
    a_n=\frac{1}{2\pi}\int_0^{2\pi}e^{i\varphi(\theta)}e^{-in\theta}d\theta.
\]
By assumption \(f\) is extremal, so we have \(|a_n|\geq \re a_n=M_n(f)\geq M_n(f_n)=2/e\). It follows that
\begin{equation}\label{eq:integral}
    \frac{2}{e}\leq \frac{1}{2\pi}\bigg|\int_0^{2\pi}e^{i\varphi(\theta)}e^{-in\theta}d\theta\bigg|.    
\end{equation}

Our strategy is to bound the integral above by a multiple of \(N/n\), and to this end we adapt the techniques of Horowitz and Ermers \cite{Ermers1990,Horowitz1978} who proved uniform bounds on \(|a_n|\) using oscillatory integral methods. First we collect some properties of \(\varphi\). The following facts are almost immediate from \eqref{eq:radial_limit} and the definition of \(\varphi\), so for brevity we omit the proof. 

\medskip

\begin{lem}\label{properties on boundary}
Let \(\varphi\) be as above and assume that \(0\leq \theta_1<\cdots<\theta_N<2\pi\). Then
\begin{enumerate}
    \item[(i)] On each \((\theta_j,\theta_{j+1})\), the functions \(\varphi\) and \(\varphi''\) are strictly increasing from \(-\infty\) to \(\infty\).
    \item[(ii)] On each \((\theta_j,\theta_{j+1})\), the function \(\varphi'\) is convex and bounded below by a positive constant.
\end{enumerate}
\end{lem}

Following \cite{Ermers1990,Horowitz1978}, for constants \(k_1>1\) and \(k_2\in(0,1)\) we set \(K_1=\{\theta\in [0,2\pi)\;|\; \varphi'(\theta)>k_1n\}\) and \(K_2=\{\theta\in [0,2\pi)\;|\; \varphi'(\theta)<k_2n\}\). On \(K_1\) the integrand in \eqref{eq:integral} is oscillatory, and we can exploit cancellation. On \(K_2\) we can integrate by parts. The next estimate given by Ermers \cite[(2.2)]{Ermers1990} shows that the untreated region \([0,2\pi)\setminus(K_1\cup K_2)\) is small.

\begin{lem}\label{lem:ermers}
Let \(K_1\) and \(K_2\) be as above. Then 
\[
    2\pi\leq  \frac{k_1+k_2}{k_2}|K_1|+|K_2|.
\]
\end{lem}

\noindent\textit{Remark}: One can considerably simplify the proofs of this inequality in \cite{Ermers1990} and \cite{Horowitz1978} using the straightforward estimate \(|\varphi(\theta)-\varphi(\vartheta)|\leq |\theta-\vartheta|\sqrt{\varphi'(\theta)\varphi'(\vartheta)}\)
for \(\theta,\vartheta\in[0,2\pi)\setminus\{\theta_1,\dots,\theta_N\}\).

Finally, we require a standard result for oscillatory integrals (see e.g. \cite[Chap. 2 Prop.2 ]{SteinMurphy1993}).

\medskip

\begin{lem}[van der Corput]\label{lem:vdC}
Let \(\Phi\) be a smooth function on an interval \(I\). Assume that \(\Phi''\) has constant sign on \(I\) and that \(|\Phi'|\geq \lambda>0\) in \(I\). Then
\[
    \bigg|\int_I e^{i\Phi(\theta)}d\theta\bigg|\leq \frac{2}{\lambda}.
\]
\end{lem}

Now we prove our first main result. In what follows, we regard all intervals as being in \(\mathbb R/2\pi\mathbb Z\).

\begin{proof}[Proof of Theorem \ref{thm:1}]
From Lemma \ref{properties on boundary}\textit{(ii)} we see that \(\varphi'\) is positive and convex on each interval \((\theta_j,\theta_{j+1})\), so if \(K_2\cap (\theta_j,\theta_{j+1})\neq\emptyset \) then it is an interval. Thus \(K_2\) is a union of at most \(N\) intervals, which we denote \(I_1,\dots, I_{m}\). Writing \(I_k=(a_k,b_k)\), we integrate by parts to compute
\[
    \int_{K_2}e^{i\varphi(\theta)}e^{-in\theta}d\theta=\frac{1}{in}\sum_{k=1}^{m}(e^{i\varphi(b_k)}e^{-inb_k}-e^{i\varphi(a_k)}e^{-ina_k})+\frac{1}{n}\sum_{k=1}^{m}\int_{a_k}^{b_k}\varphi'(\theta)e^{i\varphi(\theta)}e^{-in\theta}d\theta.
\]
Using that \(\varphi'(\theta)<k_2n\) on each \(I_k\) and that \(m\leq N\),
\[
    \bigg|\int_{K_2}e^{i\varphi(\theta)}e^{-in\theta}d\theta\bigg|\leq \frac{2m}{n}+\frac{1}{n}\int_{K_2}\varphi'(\theta)d\theta\leq \frac{2N}{n}+k_2|K_2|.
\]

Similarly, Lemma \ref{properties on boundary}\textit{(i)} shows that \(K_1\cap (\theta_j,\theta_{j+1})\) is comprised of two intervals: one where \(\varphi''\) is negative, and another where it is positive. Therefore \(K_1\) is a union of \(2N\) intervals \(J_1,\dots,J_{2N}\) on which \(\varphi''\) has constant sign. The second derivative of \(\Phi(\theta):=\varphi(\theta)-n\theta\) also has constant sign on each \(J_k\), and on these intervals we have \(\Phi'(\theta)>(k_1-1)n\). Thus by Lemma \ref{lem:vdC},
\[
    \bigg|\int_{K_1}e^{i\varphi(\theta)}e^{-in\theta}d\theta\bigg|\leq \sum_{k=1}^{2N}\bigg|\int_{J_k}e^{i\varphi(\theta)}e^{-in\theta}d\theta\bigg|\leq \sum_{k=1}^{2N}\frac{2}{(k_1-1)n}=\frac{4N}{(k_1-1)n}.
\]

Combining \eqref{eq:integral} with our bounds on the integral over \(K_1\) and \(K_2\), and using the fact that \(|e^{i\varphi(\theta)}e^{-in\theta}|=1\) for a.e. \(\theta\in [0,2\pi)\), we find that
\[
    \frac{2}{e}\leq 1-\frac{|K_1|+|K_2|}{2\pi}+\frac{N}{\pi n}+\frac{k_2|K_2|}{2\pi}+\frac{2N}{\pi(k_1-1)n}.
\]
Passing to subsets of \(K_1\) and \(K_2\) if necessary, we can assume that equality holds in Lemma \ref{lem:ermers}. Precisely, this can be done by uniformly contracting the intervals of \(K_2\) from their midpoints, and the intervals of \(K_1\) from their endpoints in \(\{\theta_1,\dots,\theta_N\}\), before proceeding with the previous estimates. Setting \(\tau=|K_1|\) gives \(|K_2|=2\pi-\frac{k_1+k_2}{k_2}\tau\). As \(|K_2|\geq 0\) we have \(\tau\in [0,\frac{2\pi k_2}{k_1+k_2}]\), and
\[
    \frac{2}{e}\leq \frac{1}{2\pi}\bigg(\frac{k_1}{k_2}-k_1-k_2\bigg)\tau+\frac{N}{\pi n}\bigg(\frac{k_1+1}{k_1-1}\bigg)+k_2.
\]
Whatever value \(\tau\) may take, the right-hand side is linear in \(\tau\), so the maximum with respect to this parameter occurs at either \(\tau=0\) or \(\tau=\frac{2\pi k_2}{k_1+k_2}\). Therefore,
\[
    \frac{2}{e}\leq \max\bigg\{k_2,\frac{ k_1}{k_1+k_2}\bigg\}+\frac{N}{\pi n}\bigg(\frac{k_1+1}{k_1-1}\bigg).
\]
For simplicity we choose \(k_2=\frac{1}{2}(-k_{1}+\sqrt{k_{1}^{2}+4k_{1}})>0\) so that \(k_2=\frac{k_{1}}{k_{1}+k_{2}}\in (0,1)\). Taking \(k_1=\frac{4}{3}\), the bound above simplifies to \(\frac{2}{e}\leq \frac{2}{3}+\frac{7N}{\pi n}\). The claimed estimate follows at once.
\end{proof}

Horowitz \cite{Horowitz1978} proved a refined version of the van der Corput lemma which can likely be used to produce a better estimate. No doubt other improvements are possible, but it is not clear whether this approach can be refined enough to get \(c=1\). Certainly more detailed estimates would be needed at several stages. 

\section{Variational methods and the proof of Theorem \ref{thm:X}}

Throughout this section we fix \(n\in\mathbb N\) and an extremal function \(f(z)=\sum_{j=0}^\infty a_jz^j\) for \(M_n\), normalized so that \(a_0,a_n>0\). Letting \(P\) and \(Q\) be as in Theorem \ref{thm:X}, we note that by \cite[Prop. 3(d)-(e)]{MartinSawyerTueroVukotic2015} it holds that \(\re P(z)> 0 \) in \(\mathbb D\) and \(\re P(\alpha_j)=0\) for \(1\leq j\leq N\). Thus, \(\theta\mapsto \re P(e^{-i\theta})\) is a non-negative trigonometric polynomial, and if we write \(\alpha_{j}=e^{-i\theta_j}\) then 
\[
    0=\frac{d}{d\theta}\bigg|_{\theta=\theta_j}\re P(e^{-i\theta})=\re\bigg\{\frac{d}{d\theta}\bigg|_{\theta=\theta_j} P(e^{-i\theta})\bigg\}=-\re \{ie^{-i\theta_j}P'(e^{-i\theta_j})\}=\im \alpha_jP'(\alpha_j).
\]

The equation \(\re P(\alpha_k)=0\) gives \(P(\alpha_k)=-\overline{P(\alpha_k)}\), and for any \(r\in\mathbb N\cup\{0\}\) it holds that \(\alpha_k^rP(\alpha_k)=-\overline{\alpha_k^{-r}P(\alpha_k)}\). Expanding \(P\), scaling this by \(\lambda_k\), and summing over \(k\), we find that
\[
    \sum_{k=1}^N\lambda_k\bigg(a_n\alpha_k^r+2\sum_{j=1}^na_{n-j}\alpha_k^{j+r}\bigg)=-\sum_{k=1}^N\lambda_k\overline{\bigg(a_n\alpha_k^{-r}+2\sum_{j=1}^na_{n-j}\alpha_k^{j-r}\bigg)}.
\]
Since \(a_n\in\mathbb R\), this simplifies to 
\[
    a_n\sum_{k=1}^N\lambda_k\alpha_k^r+\sum_{k=1}^N\sum_{j=1}^n \lambda_ka_{n-j}\alpha_k^{r+j}=-\sum_{k=1}^N\sum_{j=1}^n\lambda_k\overline{a_{n-j}}\alpha_k^{r-j}.
\]
From \eqref{extremal}, it follows that if we write \(f=e^g\) and \(g(z)=\sum_{j=0}^\infty b_jz^j\), then \(b_0=-\sum_{k=1}^N\lambda_k\) and \(\smash{b_j=-2\sum_{k=1}^N\lambda_k\alpha_k^j}\) for \(j\geq 1\). Writing \(b_j=\overline{b_{-j}}\) if \(j<0\), the preceding expression becomes
\[
    a_n\bigg(-2\sum_{k=1}^N\lambda_k\alpha_k^r\bigg)+\sum_{j=1}^na_{n-j}b_{r+j}=-\sum_{j=1}^n\overline{a_{n-j}}b_{r-j}-\overline{a_{n-r}}b_0\chi_{\{1,\dots,n\}}(r), 
\]
where \(\chi\) denotes an indicator function. Upon case-by-case inspection, this further simplifies to
\begin{equation}\label{eq:vanishing r}
    \sum_{j=0}^na_{n-j}b_{r+j}=-\sum_{j=1}^n\overline{a_{n-j}}b_{r-j}-\overline{a_{n-r}}b_0\chi_{\{0,\dots,n\}}(r),\qquad r\geq 0.
\end{equation}

Similarly, if \(\im \alpha_kP'(\alpha_k)=0\) then \(\alpha_kP'(\alpha_k)=\overline{\alpha_kP'(\alpha_k)}\) and \( \alpha_k^{r+1}P'(\alpha_k)=\overline{\alpha_{k}^{1-r}P'(\alpha_k)}\) for any \(r\geq 0 \). Expanding the polynomials in this identity gives
\[
    \sum_{j=1}^nj\alpha_k^{r+j}a_{n-j}=\overline{\sum_{j=1}^nj\alpha_k^{j-r}a_{n-j}}=\sum_{j=1}^nj\alpha_k^{r-j}\overline{a_{n-j}}.
\]
Scaling this by \(\lambda_k\), summing over \(k\), and swapping the order of summation as above, we find that
\[
    \sum_{j=1}^nja_{n-j}\bigg(-2\sum_{k=1}^N \lambda_k\alpha_k^{r+j}\bigg)=\sum_{j=1}^nj\overline{a_{n-j}}\bigg(-2\sum_{k=1}^N\lambda_k\alpha_k^{r-j}\bigg).
\]
Once again, we recognize the bracketed terms above as coefficients of \(g\), allowing us to simplify 
\begin{equation}\label{eq:derivative vanishing r}
    \sum_{j=1}^nja_{n-j}b_{r+j}=\sum_{j=1}^nj\overline{a_{n-j}}b_{r-j}+r\overline{a_{n-r}}b_0\chi_{\{1,\dots,n\}}(r),\qquad r\geq 0.    
\end{equation}
Equipped with these coefficient identities, we can prove the claimed formulas of Theorem \ref{thm:X}.

\begin{proof}[Proof of Theorem \ref{thm:X}]
For \(r\geq 0\) and \(z\in \mathbb D\), we multiply \eqref{eq:vanishing r} by \(z^{n+r}\) and sum over \(r=0,\dots, m\) for \(m\geq n\) to get
\[
    \sum_{r=0}^m\sum_{j=0}^na_{n-j}b_{r+j}z^{n+r}=-\sum_{r=0}^m\sum_{j=1}^n\overline{a_{n-j}}b_{r-j}z^{n+r}-b_0\sum_{r=0}^n\overline{a_{n-r}}z^{n+r}.
\]
Interchanging the order of summation and re-indexing the inner sums to resemble the Taylor polynomials of \(g\), we find that
\[
    \sum_{j=0}^na_{n-j}z^{n-j}\bigg(\sum_{r=j}^{m+j}b_{r}z^{r}\bigg)=-\sum_{j=1}^n\overline{a_{n-j}}z^{n+j}\bigg(\sum_{r=-j}^{m-j}b_{r}z^{r}\bigg)-b_0z^n\sum_{r=0}^n\overline{a_{n-r}}z^{r}.
\]
Sending \(m\to\infty\), the Taylor polynomials of \(g\) converge uniformly on compact subsets of \(\mathbb D\), so
\[
    \sum_{j=1}^na_{n-j}z^{n-j}\bigg(g(z)-\sum_{r=0}^{j-1}b_{r}z^{r}\bigg)=-\sum_{j=1}^n\overline{a_{n-j}}z^{n+j}\bigg(\sum_{r=1}^{j}\overline{b_{r}}z^{-r}+g(z)\bigg)-b_0\sum_{r=0}^n\overline{a_{n-r}}z^{r+n}.
\]
Rearranging to isolate \(g(z)\), we find that
\[
    g(z)\bigg(\sum_{j=0}^na_{n-j}z^{n-j}+\sum_{j=1}^n\overline{a_{n-j}}z^{n+j}\bigg)=\sum_{j=1}^na_{n-j}\bigg(\sum_{r=0}^{j-1}b_{r}z^{n-j+r}\bigg)-\sum_{j=0}^n\overline{a_{n-j}}\bigg(\sum_{r=0}^{j}\overline{b_{r}}z^{n+j-r}\bigg). 
\]

It remains to express the polynomials above in terms of known quantities. The first term on the right-hand side above can be treated by re-indexing (to group by powers of \(z\)) to get
\[
    \sum_{j=1}^na_{n-j}\bigg(\sum_{r=0}^{j-1}b_{r}z^{n-j+r}\bigg)=\sum_{r=0}^{n-1}\sum_{j=0}^{r}a_jb_{r-j}z^{r}=\sum_{r=0}^{n-1}T_r(fg)z^{r}=\frac{z^n}{2}\bigg(Q\bigg(\frac{1}{z}\bigg)-T_n(fg)\bigg).
\]
For the second term on the right, we proceed identically to compute
\[
    \sum_{j=0}^n\overline{a_{n-j}}\bigg(\sum_{r=0}^{j}\overline{b_{r}}z^{n+j-r}\bigg)=z^{2n}\sum_{j=1}^n\overline{a_{n-j}}\bigg(\sum_{r=0}^{j-1}\overline{b_{r}}(z^{-1})^{n-j+r}\bigg)+\overline{T_n(fg)}z^{n}=\frac{z^n}{2}\bigg(\overline{Q(\overline{z})}+\overline{T_n(fg)}\bigg).
\]
As \(\re T_n(fg)=\sum_{k=0}^N\lambda_k \re P(\alpha_k)=0\), the right-hand side simplifies to give 
\[
    g(z)\bigg(\sum_{j=0}^na_{n-j}z^{n-j}+\sum_{j=1}^n\overline{a_{n-j}}z^{n+j}\bigg)=\frac{z^n}{2}\bigg(Q\bigg(\frac{1}{z}\bigg)-\overline{Q(\overline{z})}\bigg). 
\]

Similarly, to determine the polynomial on the left, we use that \(a_n\in\mathbb R\) to compute
\[
    \sum_{j=0}^na_{n-j}z^{n-j}=\frac{z^n}{2}\bigg(P\bigg(\frac{1}{z}\bigg)+a_n\bigg),\qquad \sum_{j=1}^n\overline{a_{n-j}}z^{n+j}=\frac{z^n}{2}\bigg(\overline{P(\overline{z})}-a_n\bigg).
\]
Adding these expressions gives
\[
    \sum_{j=0}^na_{n-j}z^{n-j}+\sum_{j=1}^n\overline{a_{n-j}}z^{n+j}=\frac{z^n}{2}\bigg(P\bigg(\frac{1}{z}\bigg)+\overline{P(\overline{z})}\bigg).
\]
The desired formula for \(g(z)\) follows at once.

The procedure above can be repeated using equation \eqref{eq:derivative vanishing r} to recover the formula for \(zg'(z)\). Adding \(r\) times \eqref{eq:vanishing r} to \eqref{eq:derivative vanishing r} first gives us a more convenient formula,
\[
    \sum_{j=0}^n(r+j)a_{n-j}b_{r+j}=\sum_{j=1}^n(j-r)\overline{a_{n-j}}b_{r-j},\qquad r\geq 0.
\]
Scaling this by \(z^{n+r}\) and summing over \(r\), we find that for \(m\in\mathbb N\),
\[
    \sum_{r=0}^m\sum_{j=0}^n(r+j)a_{n-j}b_{r+j}z^{n+r}=\sum_{r=0}^m\sum_{j=1}^n(j-r)\overline{a_{n-j}}b_{r-j}z^{n+r}.
\]
Interchanging the order of summation and re-indexing shows that if \(m>n\) then,
\begin{gather*}
    \sum_{j=1}^na_{n-j}z^{n-j+1}\bigg(\sum_{r=j}^{m+j}rb_{r}z^{r-1}\bigg)=-\sum_{j=1}^n\overline{a_{n-j}}z^{n+j}\bigg(\sum_{r=0}^{j-1}(r-j)b_{r-j}z^{r-j}+z\sum_{r=1}^{m-j}rb_rz^{r-1}\bigg).
\end{gather*}
Finally, sending \(m\to\infty\), we obtain a formula involving \(g'\), 
\[
    \sum_{j=1}^na_{n-j}z^{n-j+1}\bigg(g'(z)-\sum_{r=1}^{j-1}rb_{r}z^{r-1}\bigg)=-\sum_{j=1}^n\overline{a_{n-j}}z^{n+j}\bigg(\sum_{r=0}^{j-1}(r-j)b_{r-j}z^{r-j}+zg'(z)\bigg).
\]
Isolating \(g'(z)\), we get
\[
    zg'(z)=\frac{\displaystyle\sum_{j=1}^n\bigg(a_{n-j}z^{n-j}\bigg(\sum_{r=1}^{j-1}rb_{r}z^{r}\bigg)+\overline{a_{n-j}}z^{n+j}\bigg(\sum_{r=1}^{j}r\overline{b_{r}}z^{-r}\bigg)\bigg)}{\displaystyle \sum_{j=0}^na_{n-j}z^{n-j}+\sum_{j=1}^n\overline{a_{n-j}}z^{n+j}}.
\]

The denominator is \(\smash{\frac{1}{2}z^n(P(1/z)+\overline{P(\overline{z})})}\) by our earlier work. The numerator can also be simplified considerably using the fact that \(ja_j=\sum_{k=1}^{j} kb_ka_{j-k}\) for \(j\in\mathbb N\), since \(f'=fg'\). For instance, the first term simplifies to
\[
    \sum_{j=1}^na_{n-j}z^{n-j}\bigg(\sum_{r=1}^{j-1}rb_{r}z^{r}\bigg)=\sum_{j=1}^n\bigg(\sum_{r=0}^jrb_ra_{j-r}\bigg)z^j-na_nz^n=\sum_{j=1}^{n-1}ja_j z^j.
\]
Similarly, 
\[
    \sum_{j=1}^n\overline{a_{n-j}}z^{n+j}\bigg(\sum_{r=1}^{j}r\overline{b_{r}}z^{-r}\bigg) =\sum_{j=1}^n\overline{\bigg(\sum_{r=0}^{j}ra_{j-r}b_{r}\bigg)}z^{2n-j}=\sum_{j=1}^nj\overline{a_j}z^{2n-j}.
\]
With these calculations, the numerator in our formula above for \(g'(z)\) simplifies to 
\[
    \sum_{j=1}^na_{n-j}z^{n-j}\bigg(\sum_{r=1}^{j-1}rb_{r}z^{r}\bigg)+\sum_{j=1}^n\overline{a_{n-j}}z^{n+j}\bigg(\sum_{r=1}^{j}r\overline{b_{r}}z^{-r}\bigg)
    =\sum_{j=1}^{n}ja_j z^j+\sum_{j=1}^nj\overline{a_j}z^{2n-j}-na_nz^n.
\]
Our goal is to express this in terms of \(P\). To this end we recognize that
\[
    \sum_{j=1}^{n}ja_j z^j=\frac{z^n}{2}\bigg(2\sum_{j=0}^{n}(n-j)a_{n-j} z^{-j}\bigg)=\frac{z^n}{2}\bigg(nP\bigg(\frac{1}{z}\bigg)-\frac{1}{z}P'\bigg(\frac{1}{z}\bigg)+na_n\bigg),
\]
and
\[
    \sum_{j=1}^nj\overline{a_j}z^{2n-j}=\frac{z^n}{2}\bigg(2\sum_{j=0}^n(n-j)\overline{a_{n-j}}z^{j}\bigg)=\frac{z^n}{2}\big(n\overline{P(\overline{z})}-z\overline{P'(\overline{z})}+na_n\big).
\]
Using these simplifications, we find that
\[
    zg'(z)
    =\frac{nz^nP(1/z)+nz^n\overline{P(\overline{z})}-(z^{n-1}P'(1/z)+z^{n+1}\overline{P'(\overline{z})})}{z^nP(1/z)+z^n\overline{P(\overline{z})}}=n-\frac{z^{n-1}P'(1/z)+z^{n+1}\overline{P'(\overline{z})}}{z^nP(1/z)+z^n\overline{P(\overline{z})}}.
\]
This completes the proof of Theorem \ref{thm:X}.
\end{proof}

The preceding expressions simplify even further if we restrict attention to \(\partial \mathbb D\). If \(|z|=1\) then \(z^{-1}=\overline{z}\), and our expression of the first formula reduces to \(g(z)=i\im Q(\overline{z})/\re P(\overline{z})\). Recall that we may also write \(g(e^{i\theta})=i\varphi(\theta)\), giving the first formula of Corollary \ref{cor:formulas for varphi}. Similarly, we observe that \(\varphi'(\theta)=e^{i\theta}g'(e^{i\theta})\), and the second formula of Corollary \ref{cor:formulas for varphi} follows in identical fashion. 

In passing, we also note that if \(g_j(z)=\frac{1+\alpha_jz}{1-\alpha_jz}\) then \(M_n(fg_jg_k)=0\) for \(1\leq j,k\leq N\). This is easily verified by perturbing two atoms simultaneously in the manner of \cite[\S1.4]{MartinSawyerTueroVukotic2015}. So,
\[
    0= \sum_{k=1}^N\lambda_k M_n(ifg_jg_k)=M_n(ifgg_j)=\re iT_n(fgg_j)=-\im Q(\alpha_j).
\]
Thus, \(\im Q(\alpha_j)=0\) for \(1\leq j\leq N\). 

It remains to justify Corollary \ref{cor: lower bound on N}.

\begin{proof}[Proof of Corollary \ref{cor: lower bound on N}]
Define \(T(\theta):=\re \{e^{-i\theta}P'e^{-i\theta}\}\). From the second formula in Corollary \ref{cor:formulas for varphi}, we observe that \(\varphi'(\theta)=n\) if and only if \(T(\theta)=0\). As established by Lemma \ref{properties on boundary}, the function \(\varphi'\) is strictly convex on each interval \((\theta_j,\theta_{j+1})\), so \(\varphi'(\theta)=n\) has at most two solutions in each of these intervals, hence at most \(2N\) solutions in all of \(\mathbb T\). If \(Z=\{\theta\in [0,2\pi)\;|\; T(\theta)=0\}\), then it follows that \(\#Z\leq 2N\). 

Next, let \(\Gamma(\theta)=e^{i\theta}P'(e^{i\theta})\) and assume that \(zP'(z)\) has \(m\) zeros in \(D\), counting multiplicities. By the argument principle, the winding number of \(\Gamma\) around the origin is exactly \(m\). Each time \(\Gamma\) winds around the origin, it intersects the imaginary axis twice, meaning that there are (at least) \(2m\) distinct points \(\theta_1,\dots,\theta_{2m}\in [0,2\pi)\) where \(\re \Gamma(\theta_j)=0\). As \(\re \Gamma=T\), we see that \(T\) has at least \(2m\) distinct zeros, so \(\#Z\geq 2m\). Combining with our bound above, we see that \(N\geq m\). 

Finally, note that \(P'(z)\) is a degree \(n-1\) polynomial. If it only vanishes in \(D\), then \(zP'(z)\) has \(n\) zeros in \(D\), counting multiplicities. With the aforementioned bound, this forces \(N=n\). 
\end{proof}

\section{Proper holomorphic maps and the proof of Theorem \ref{thm:4}}

Now we explore the holomorphic invariants of extremal functions. Our main tools are a result about the invariants of Blaschke products on the circle from \cite{CassierChalendar2000}, and a characterization of the proper holomorphic self-maps of \(\mathbb D\).

\medskip

\begin{thm}[Cassier–Chalendar \cite{CassierChalendar2000}]\label{thm:C-C}
Let \(B\) be a finite Blaschke product of degree \(N\). The set of continuous functions \(\Phi:\partial\mathbb D\to \partial\mathbb D\) such that \(B|_{\partial\mathbb D}\circ\Phi\equiv B|_{\partial\mathbb D}\) is a cyclic group under composition of order \(N\). 
\end{thm}

It follows that if \(\Phi\) is a continuous invariant of \(B\) on the circle, then \(\Phi^{\circ N}=\textrm{id}\). Here \(\Phi^{\circ N}\) denotes the \(N\)-fold composition of \(\Phi\) with itself, and we use this notation throughout the section. 

\medskip

Next, a holomorphic map \(\psi: \mathbb D\to\mathbb D\) is called proper if \(\psi^{-1}(K)\) is compact in \(\mathbb D\) for any compact \(K\subset\mathbb D\). It is well-known that a self-map of \(\mathbb D\) is proper if and only if it is a finite Blaschke product \cite[\S1]{Bedford1984}. For completeness, we prove this equivalence.

\medskip

\begin{lem}\label{proper blaschke}
A holomorphic self-map of \(\mathbb D\) is proper if and only if it is a finite Blaschke product. 
\end{lem}

\begin{proof}
First let \(\psi\) be a finite Blaschke product. Given a compact set \(K\subset\mathbb D\), we can choose \(r<1\) such that if \(w\in K\) then \(|w|\leq r\). Then
\[
    \psi^{-1}(K)\subseteq \psi^{-1}(\{w\in \mathbb D\;|\; |w|\leq r\})=\{z\in{\mathbb D}\;|\; |\psi(z)|\leq r\}.
\]
As \(\psi\) is a finite Blaschke product it holds that \(|\psi(z)|<1\) if and only if \(|z|<1\), so the right-hand side above is contained in \(\mathbb D\). Therefore \(\psi^{-1}(K)\subset \mathbb D\). Furthermore, as \(\psi\) is continuous in \(\mathbb D\) and \(K\) is compact, \(\psi^{-1}(K)\) is also closed and therefore compact. It follows that \(\psi\) is proper.

On the other hand, if \(\psi\) is proper then for any \(r<1\) the set \(\psi^{-1}(\{|w|\leq r\})\) is compact in \(\mathbb D\), so there exists \(\rho<1\) such that if \(|\psi(z)|\leq r\) then \(|z|\leq \rho\). Thus if \(|z|>\rho\) then \(|\psi(z)|\geq r\), and we see that \(|\psi(z)|\to 1\) uniformly as \(|z|\to 1\).  Also note that \(\psi^{-1}(\{0\})\) is compact and discrete. Otherwise, \(\psi\equiv 0\) by the identity theorem and \(\mathbb D=\psi^{-1}(\{0\})\) would be compact, a contradiction. Thus \(\psi^{-1}(\{0\})\) is finite, say \(\psi^{-1}(\{0\})=\{z_1,\dots, z_N\}\). Now define 
\[
    B(z)=\prod_{j=1}^N\frac{z-z_j}{1-\overline{z_j}z}
\]
and let \(\eta=\psi/B\). Then \(\eta\) is non-vanishing in \(\mathbb D\) and \(|\eta(z)|\to 1\) uniformly as \(|z|\to 1\). The same is true of \(1/\eta\), so \(\eta\) is a unimodular constant by the maximum modulus principle. Therefore \(\psi=\eta B\), showing that \(\psi\) is a finite Blaschke product. 
\end{proof}

It will also be useful to note that a composition \(B_1\circ B_2\) of finite Blaschke products \(B_1\) and \(B_2\) of respective degrees \(N_1\) and \(N_2\) is a Blaschke product of degree \(N_1N_2\) (see e.g. \cite[Thm. 3.12]{GarciaMashreghiRoss2018}). Using this and the aforementioned facts, we characterize the invariants of extremal functions. 

\begin{proof}[Proof of Theorem \ref{thm:4}]
Let \(f\) be extremal for \(M_n\) and suppose that \(N\geq 3\)  (here \(N\) is as in \eqref{extremal}). If \(f\circ \psi=f\) then \(\log f\circ\psi=\log f+2\ell\pi i\) for some \(\ell\in\mathbb Z\). There exists \(t>0\) and a Blaschke product \(h\) of degree \(N\) such that \(\log f=t\frac{h-1}{h+1}\) (this is justified in Lemma \ref{lem:h is Blaschke} below). Therefore in all of \(\mathbb D\) we have
\begin{equation}\label{kill c}
    \frac{h\circ\psi-1}{h\circ\psi+1}=\frac{h-1}{h+1}+2ic,
\end{equation}
where \(c=\ell/2t\in\mathbb R\). Solving the equation above to write \(h\circ\psi\) in terms of \(h\) and \(c\), we find that
\[
    h\circ\psi=\zeta\frac{h-b}{1-\overline{b}h},\quad\textrm{where}\quad b=-\frac{ic}{1+ic},\quad \zeta=\frac{1+ic}{1-ic}.
\]
In particular we note that \(b\in \mathbb D\) and \(\zeta\in\partial\mathbb D\), so
\[
    \omega(z):=\zeta\frac{z-b}{1-\overline{b}z}\in\textrm{Aut}(\mathbb D),
\]
and it is also easily verified that \(\omega(-1)=-1\). Writing \(h\circ\psi=\omega\circ h\), we note that \(\omega\circ h\) is a degree \(N\) Blaschke product. Now given a compact set \(K\subset \mathbb D\), we observe that \(\psi^{-1}(K)\) is closed by continuity in \(\mathbb D\). Moreover, 
\[
    \psi^{-1}(K)\subseteq \psi^{-1}(h^{-1}(h(K)))=(h\circ \psi)^{-1}(h(K))=(\omega\circ h)^{-1}(h(K)).
\]
The set \(h(K)\) is compact since \(h\) is continuous, and by Lemma \ref{proper blaschke} \(\omega\circ h\) is proper, so it follows that \((\omega\circ h)^{-1}(h(K))\) is compact in \(\mathbb D\). Thus \(\psi\) is proper, so it is a finite Blaschke product by Lemma \ref{proper blaschke}. If \(\psi\) has degree \(N_0\) then we deduce that \(NN_0=N\), so \(N_0=1\) and \(\psi\in\mathrm{Aut}(\mathbb D)\).

It remains to show that \(\psi\) is periodic. From \eqref{extremal} it is evident that \(h(z)=-1\) if and only if \(z\in A:=\{\overline{\alpha_1},\dots,\overline{\alpha_N}\}\). Moreover, if \(j\in\{1,\dots, N\}\) then \(h(\psi(\overline{\alpha_j}))=(\psi\circ h)(\overline{\alpha_j})=\omega(-1)=-1\), showing that \(\psi(\overline{\alpha_j})\in A\). As \(\psi\) is also injective, \(\psi|_A\) is a permutation. Thus there exists \(m\in\mathbb N\) such that \(\psi^{\circ m}|_{A}=\textrm{id}|_A\) (decompose \(\psi|_A\) into cycles, and take \(m\) as the least common multiple of the cycle lengths). As \(\#A=N\geq 3\) it follows that \(\psi^{\circ m}(z)=z\) in all of \(\mathbb D\), since a non-identity automorphism can fix at most two points on \(\overline{\mathbb D}\). 

Using the fact that \(\psi^{m}=\mathrm{id}\), we will show that \(\psi\) has a fixed point in \(\mathbb D\). To this end recall that the matrix group \(SU(1,1)\) is isomorphic to \((\mathrm{Aut}(\mathbb D),\circ)\) by the map
\[
    SU(1,1)\ni M=\begin{pmatrix}
        \alpha & \beta\\
        \overline{\beta} & \overline{\alpha}
    \end{pmatrix}\quad\xrightarrow{\sim}\quad\mu_M(z):=\frac{\alpha z+\beta}{\overline{\beta}z+\overline{\alpha}}\in \mathrm{Aut}(\mathbb D).
\]
In particular, any element of \(\mathrm{Aut}(\mathbb D)\) is of the form \(\mu_M\) for some \(\alpha,\beta\in\mathbb C\) and \(M\) as above. Thus there exists \(M\in SU(1,1)\) such that \(\psi=\mu_M\), and since \(\psi^{\circ m}=\textrm{id}\) we have
\[
    \mu_I=\mathrm{id}=\psi^{\circ m}=\mu_{M}^{\circ m}=\mu_{M^m}.
\]
Thus \(M^m=I\), so the minimal polynomial of \(M\) divides \(z^m-1\) and it has no repeat roots. This implies that \(M\) is diagonalizable. If \(\lambda_1,\,\lambda_2\) are the eigenvalues of \(M\) then \(\lambda_1^m=\lambda_2^m=1\) and \(\lambda_1\lambda_2=\det M=1\), giving \(\lambda_2=\overline{\lambda_1}\). 

If \(\lambda_1\neq \lambda_2\) and \(v_1=(z_1,z_2)^T\in\mathbb C^2\) is an eigenvector for \(\lambda_1\), then it is easily verified that \(v_2=(\overline{z_2},\overline{z_1})^T\) is an eigenvector for \(\lambda_2\). The vectors \(v_1\) and \(v_2\) are also linearly independent since \(M\) is diagonalizable, so \(|z_1|^2-|z_2|^2\neq 0\). Assuming that \(|z_1|>|z_2|\) (otherwise, swap eigenvalues and conjugate) we can rescale so that \(|z_1|^2-|z_2|^2=1\). Assuming this done, take
\[
    P=\begin{pmatrix}
        z_1 & \overline{z_2}\\
        z_2 & \overline{z_1}
    \end{pmatrix}\in SU(1,1),\qquad \mu_P(z)=\frac{z_1z+\overline{z_2}}{z_2z+\overline{z_1}}\in\mathrm{Aut}(\mathbb D)
\]
and \(D=\mathrm{diag}(\lambda_1,\lambda_2)\), so that \(M=PDP^{-1}\). Then \(\psi=\mu_M=\mu_P\circ \mu_D\circ \mu_P^{-1}\), and since \(\mu_D\) is a rotation (hence it fixes zero) we find that \(\psi\) fixes \(a:=\overline{z_2}/\overline{z_1}\in \mathbb D\). In case \(\lambda_1=\lambda_2\) note that \(M\) is conjugate to \(\pm I\), so \(M=\pm I\) and \(\psi(z)=\mu_M(z)=\mu_{\pm I}(z)=z\) and \(\psi\) fixes \(a=0\).

Now it follows by evaluating \eqref{kill c} at \(z=a\) that \(c=0\) and \(\psi\) is an invariant of \(h\). Thus, \(\Phi=\psi|_{\partial\mathbb D}\) is a continuous invariant of \(h|_{\partial \mathbb D}\), and since \(h\) has degree \(N\) it follows from Theorem \ref{thm:C-C} that \(\Phi^{\circ N}=\textrm{id}\). By the maximum modulus principle then, \(\psi^{\circ N}=\textrm{id}\) in all of \(\mathbb D\). If we set
\[
    \psi_a(z)=\frac{z-a}{1-\overline{a}z}\in\mathrm{Aut}(\mathbb D),
\]
then the composition \(\psi_a\circ \psi\circ\psi_{-a}:\mathbb D\to\mathbb D\) is an automorphism that fixes zero, hence it is a rotation, and for some \(\xi\in\partial{\mathbb D}\) we can write \(\psi(z)=\psi_{-a}(\xi \psi_a(z)). \) Then \(\psi^{\circ N}=\textrm{id}\) if and only if \(z=\psi_{-a}(\xi^N \psi_a(z))\) and \(\psi_a(z)=\xi^N\psi_a(z)\) in \(\mathbb D\), meaning that \(\xi^N=1\).

Finally, in case \(N=2\) we argue that either \(\psi^{\circ 2}=\mathrm{id}\), in which case the argument above carries through, or \(\psi\) is a hyperbolic automorphism. If \(\#A=2\) and \(\psi\) does not fix \(A\) then no point of \(A\) is fixed under \(\psi\). By the Brouwer fixed point theorem, there exists a point \(a\in\overline{\mathbb D}\setminus A\) such that \(\psi(a)=a\). Then \(\psi^{\circ 2}\) fixes \(a\) as well as both points in \(A\), so \(\psi^2=\mathrm{id}\) and the argument above reduces \(\psi\) to the desired form. Otherwise, \(\psi\) fixes \(A=\{\overline{\alpha_1},\overline{\alpha_2}\}\) as claimed. 
\end{proof}

Equipped with Theorem \ref{thm:4}, Corollary \ref{cor:4.1} follows easily.

\begin{proof}[Proof of Corollary \ref{cor:4.1}]
If \(\psi\) is an invariant of \(f\) then \(\psi\in\mathrm{Aut}(\mathbb D)\), and since it also fixes zero we deduce that \(\psi\) is a rotation and \(\psi^{\circ N}=\mathrm{id}\). Writing \(\psi(z)=\xi z\) we have \(\xi^N=1\) and
\[
    f(z)=\frac{1}{N}\sum_{j=1}^Nf(\xi^j z)=\sum_{j=0}^\infty a_{Nj}z^{Nj}.
\]
If \(\mathrm{gcd}(n,N)=1\) then the identity above implies that \(a_n=0\), a contradiction since \(|a_n|\geq 2/e\). 
\end{proof}

\section{Properties of Blaschke products and the proof of Theorem \ref{thm:3}}

If \(f\in\mathcal B_0\) is normalized so that \(f(0)>0\), then taking \(t=-\log f(0)>0\) we may write \(f=\exp(t\frac{h-1}{h+1})\), where \(h\) is an inner function satisfying \(h(0)=0\). In fact, more can be said.

\medskip

\begin{lem}\label{lem:h is Blaschke}
Let \(f=\exp(t\frac{h-1}{h+1})\) be extremal for \(M_n\). Then \(h\) is a Blaschke product of degree \(N\).
\end{lem}

\begin{proof}
First, we establish that \(h\) has a unimodular radial limit everywhere. It follows from \eqref{eq:radial_limit} that if \(\theta\not\in\{\theta_1,\dots,\theta_N\}\) then the limit of \(\displaystyle\lim_{r\to 1^-}\log f(re^{i\theta})\) exists and 
\[
    h(re^{i\theta})=\frac{t+\log f(re^{i\theta})}{t-\log f(re^{i\theta})}\overset{\,\,r\to 1^-}{\longrightarrow}\frac{t+i\varphi(\theta)}{t-i\varphi(\theta)}=\exp\bigg(2i\arctan\bigg(\frac{\varphi(\theta)}{t}\bigg)\bigg)\in\partial \mathbb D.
\]
On the other hand, if \(\theta\in\{\theta_1\dots,\theta_N\}\) then \(|f(re^{i\theta})|\to 0\). In this case we point out that 
\[
      |1+h(re^{i\theta_j})|^2\leq \frac{|1+h(re^{i\theta_j})|^2}{1-|h(re^{i\theta_j})|^2}=\frac{t}{|\log|f(re^{i\theta_j})||} \overset{\,\,r\to 1^-}{\longrightarrow} 0,
\]
showing that \(h(r\overline{\alpha_j})\to -1\). 

Next, define degree \(N\) polynomials \(q(z)=\prod_{j=1}^N(1-\alpha_jz)\) and \(r(z)=q(z)\log f(z)\) and write
\[
    h(z)=\frac{t+\log f(z)}{t-\log f(z)}=\frac{tq(z)+r(z)}{tq(z)-r(z)}.
\]
This shows that \(h\) is a rational function with at most \(N\) zeros. As \(h\in \mathrm{Hol}(\mathbb D)\) and \(|h(e^{i\theta})|=1\) for all \(\theta\in\mathbb T\), we find that \(h\) is holomorphic in a neighbourhood of \(\overline{ \mathbb D}\). 

Let \(z_1,\dots, z_m\) denote these zeros of \(h\) in \(\mathbb D\), repeated by multiplicity, and observe that \(h(z)/\prod_{j=1}^m\frac{z-z_j}{1-\overline{z_j}z}\) defines a non-vanishing holomorphic function in a neighbourhood of \(\overline{\mathbb D}\) with unit modulus on \(\partial \mathbb D\). Hence it is constant, showing that \(\smash{h(z)=\xi\prod_{j=1}^m\frac{z-z_j}{1-\overline{z_j}z}}\) for \(\xi\in\partial \mathbb D\) and \(m\leq N\). Writing \(h=p_1/p_2\) for degree \(m\) polynomials \(p_1,\,p_2\), observe that since \(h(z)=-1\) has \(N\) distinct solutions, the degree \(m\) polynomial \(p_1+p_2\) has \(N\) distinct roots, so \(m\geq N\). 
\end{proof}

Several of the equivalences of Theorem \ref{thm:3}  rely on properties special to Blaschke products. The only possibly non-standard one we require is stated by Garcia, Mashreghi \& Ross \cite[Thm. 8.2]{GarciaMashreghiRoss2018}.

\medskip

\begin{lem}\label{lem:Blasckhe derivative zeros}
Let \(h\) be a finite Blaschke product of degree \(N\). Then \(h'\) has \(N-1\) zeros in \(\mathbb D\), counting multiplicities.
\end{lem}

Now we are equipped to prove Theorem \ref{thm:3}. 

\begin{proof}[Proof of Theorem \ref{thm:3}]
If the Krzy\.z conjecture is true then \(f=f_n\), and conditions \textit{(1)-(5)} are immediate. Thus, it suffices to show that each of these conditions implies the conjecture. 

It is convenient to start with \textit{(5)}, which uses ideas independent of the other conditions. If \(a_j=0\) for \(1\leq j\leq \lceil\frac{n-2}{3}\rceil\) then 
\begin{equation}\label{eq:f'h'}
    f'(z)=\frac{2tf(z)h'(z)}{(h(z)+1)^2}
\end{equation}
vanishes to order \(\lceil \frac{n-2}{3}\rceil\) at zero. Hence so too does \(h'\), since \(f\neq 0\) and \(h\neq -1\) in \(\mathbb D\). It follows that \(h\) vanishes to order \(\lceil\frac{n-2}{3}\rceil+1\) at zero since \(h(0)=0\), so for \(k\geq \lceil\frac{n-2}{3}\rceil+1\geq \frac{n+1}{3}\) and a Blaschke product \(B\) we can write \(h(z)=z^kB(z)\). In particular we note that \(h(z)^3\) vanishes to order \(n+1\) at zero. Now it is useful to observe that for \(t>0\) and \(z\in \mathbb D\), 
\[
    \exp\bigg(t\frac{z-1}{z+1}\bigg)=e^{-t}+2te^{-t}z+2t(t-1)e^{-t}z^2\mod z^3.
\]
Thus, thanks to the vanishing order of \(h\), we find that
\[
    f(z)=e^{-t}+2te^{-t}h(z)+2t(t-1)e^{-t}h(z)^2\mod z^{n+1}.
\]
Write \(h(z)=\sum_{j=1}^\infty c_jz^j\) and note that \(\sum_{j=1}^\infty|c_j|^2\leq 1\) (in fact, equality holds) since \(|h|\leq 1\) in \(\mathbb D\). Using this, we estimate the \(n\)th Taylor coefficient of \(h^2\) via the Cauchy Schwarz inequality,
\[
    \bigg|\sum_{j=0}^{n}c_jc_{n-j}\bigg|=\bigg|\sum_{j=1}^{n-1}c_jc_{n-j}\bigg|\leq \bigg(\sum_{j=1}^{n-1}|c_j|^2\bigg)^\frac{1}{2}\bigg(\sum_{j=1}^{n-1}|c_{n-j}|^2\bigg)^\frac{1}{2}=\sum_{j=1}^{n-1}|c_j|^2\leq 1-|c_n|^2.
\]
Therefore, the \(n\)th coefficient of \(f\) is bounded by
\[
    |a_n|=2te^{-t}\bigg|c_n+(t-1)\sum_{j=1}^{n-1}c_jc_{n-j}\bigg|\leq 2te^{-t}(|c_n|+|t-1|(1-|c_n|^2)).
\]

It remains to optimize over \(t\). As \(t\geq -\log(\sqrt{2}-1)\approx 0.88\) for extremal functions (see e.g. \cite[\S2]{Peretz1991}) it suffices to consider \(t\geq \frac{1}{2}\). First we point out that \(|c_n|+\frac{1}{2}(1-|c_n|^2)\leq 1\) for \(|c_n|\in[0,1]\), with equality if and only if \(|c_n|=1\). Therefore for \(t\in[-\frac{1}{2},\frac{3}{2}]\) we have \(|t-1|\leq \frac{1}{2}\) and
\[
    |a_n|\leq 2te^{-t}(|c_n|+|t-1|(1-|c_n|^2))\leq 2te^{-t}\bigg(|c_n|+\frac{1}{2}(1-|c_n|^2)\bigg)\leq 2te^{-t}\leq \frac{2}{e},
\]
with equality throughout if and only if \(|c_n|=1\) and \(t=1\). If \(t\geq \frac{3}{2}\) then we have the bound \(|a_n|\leq 2te^{-t}(|c_n|+(t-1)(1-|c_n|^2))\). Maximizing over \(|c_n|\in[0,1]\), assuming that \(t\geq \frac{3}{2}\), gives
\[
    |a_n|\leq 2te^{-t}\bigg(\frac{1}{4(t-1)}+t-1\bigg).
\]
The maximum of this function over \(t\geq \frac{3}{2}\) occurs at the largest real root \(t_0\) of the polynomial \(4t^{4}-20t^{3}+33t^{2}-21t+5\). It is not difficult to check that \(t_0\geq 2\), so if \(t\geq \frac{3}{2}\) then
\[
    |a_n|\leq 2t_0e^{-t_0}\bigg(\frac{1}{4(t_0-1)}+t_0-1\bigg)\leq 2t_0e^{-t_0}\bigg(t_0-\frac{3}{4}\bigg)\leq \max_{t\geq 2}\bigg(2te^{-t}\bigg(t-\frac{3}{4}\bigg)\bigg).
\]
Computing the maximum on the right shows that if \(t\geq \frac{3}{2}\) then 
\[
    |a_n|\leq \frac{8+\sqrt{73}}{2}e^{-\frac{11+\sqrt{73}}{8}}<\frac{2}{e}.
\]
Regardless of the value of \(t\) then, we have \(|a_n|\leq2/e\) and equality holds if and only if \(t=1\) and \(|c_n|=1\). If \(|c_n|=1\) then \(|c_j|=0\) for all \(j\neq n\) and \(h(z)=\xi z^n\) for some \(\xi\in\partial\mathbb D\), so
\[
    f(z)=\exp\bigg(t\frac{h(z)-1}{h(z)+1}\bigg)=\exp\bigg(\frac{\xi z^n-1}{\xi z^n+1}\bigg).
\]
The \(n\)th coefficient of this function is \(2\xi/e\) so we deduce that \(\xi=1\) and \(f=f_n\). Thus if \textit{(5)} holds for any extremal function, then the Krzy\.z conjecture is true. 


\medskip

Next assume \textit{(1)} and observe that the \(n\)th coefficient of \(f(e^{i\theta}z)\) is \(e^{in\theta}a_n\). As \(f(z)\) and \(f(e^{i\theta}z)\) are both extremal it follows that \(\re e^{in\theta}a_n=\re a_n\). Furthermore, extremality implies that \(\lambda=0\) is a local maximum of the map \(F(\lambda)= M_n(f(e^{i\lambda}z))=\re e^{in\lambda}a_n\), so \(F'(0)=0\). Equivalently, \(\re ina_n=0\) and \(\im a_n=0\). Thus, \(a_n=\re e^{in\theta}a_n=a_n\cos(n\theta)\) and \(\cos(n\theta)=1\). As \(0<\theta<4\pi /n\) by assumption, we find that \(\theta=2\pi/n\). It follows that
\[
    f(z)=\frac{1}{n}\sum_{j=1}^nf(e^{2\pi i i/n}z)=a_0+a_nz^n\mod z^{2n},
\]
and the conjecture again follows from \textit{(5)}. 

\medskip

To conclude the proof, we show that \textit{(2)-(4)} all imply that \(h\) is a monomial, and then we show that this implies the conjecture. If \textit{(2)} holds then since \(z\to t\frac{z-1}{z+1}\) is a biholomorphism from \(\mathbb D\) to the left half-plane, we see that \textit{(2)} is equivalent to the statement that \(h(z)=h(0)\) if and only if \(z=0\). As \(h\) is degree \(N\) Blaschke product by Lemma \ref{lem:h is Blaschke}, it follows that \(h(z)=\xi  z^N\).

Similarly, if \textit{(3)} holds, then since \(f'=0\) if and only if \(h'=0\) by \eqref{eq:f'h'}, \(h'\) is non-vanishing away from \(z=0\). As \(h'\) has exactly \(N-1\) zeros in \(\mathbb D\) (counting multiplicities) by Lemma \ref{lem:Blasckhe derivative zeros}, we deduce that \(h'\) vanishes to order \(N-1\) at zero. Furthermore \(h(0)=0\), so \(h\) vanishes to order \(N\) at zero, and since \(h\) is a degree \(N\) Blaschke product this implies again that \(h(z)=\xi z^N\). If \textit{(4)} holds, then once again \(h'\) vanishes to order \(N-1\) at zero, and the same implication holds. 

Thus, if any of \textit{(2)-(4)} hold then \(h(z)=\xi z^N\). Taylor expanding \(\exp(t\frac{z-1}{z+1})=\sum_{j=0}^\infty \beta_j(t)z^j\), it follows that 
\[
    f(z)=\exp\bigg(t\frac{h(z)-1}{h(z)+1}\bigg)=\sum_{j=0}^\infty \beta_j(t)\xi^j z^{Nj}.
\]
If \(N\) does not divide \(n\) then \(a_n=0\), a contradiction since \(a_n\geq M_n(f_n)=2/e\). Thus \(n=jN\) for some \(j\in\mathbb N\) and we see that \(a_n=\xi^n\beta_j(t)=|\beta_j(t)|\), where the second equality holds by maximizing over \(\xi\). It is shown by Lewandowski \& Szynal \cite[(7)]{LewandowskiSzynal1998}, and separately by Prokhorov \cite{Prokhorov2001}, that the coefficient functions \(\beta_j(t)\) can be written in terms of generalized Laguerre polynomials, 
\[
    \beta_j(t)=(-1)^je^{-t}L_j^{(-1)}(2t),\qquad \textrm{where}\qquad L_n^{(\alpha)}(x)=\sum_{k=0}^{n}\binom{n+\alpha}{n-k}\frac{\left(-x\right)^{k}}{k!}.
\]
Rooney \cite{Rooney1985} showed that for any \(j\geq 0\), \(t\geq 0\) and \(\alpha\leq -\frac{1}{2}\), the Laguerre polynomials \(L_j^{(\alpha)}\) satisfy 
\[
    |L_j^{(\alpha)}(t)|\leq \frac{\sqrt{(2j)!}}{2^{j+\alpha+1/2}j!}e^{t/2}.
\]
Taking \(\alpha=-1\) and using the formula for \(\beta_j\) above shows that \(\sup_{t\geq 0}|\beta_j(t)|\leq \sqrt{2\left(2j\right)!}/2^{j}j!\). For \(j\geq 5\) this upper bound is strictly smaller than \(2/e\). One can also optimize \(|\beta_j(t)|\) for \(2\leq j\leq 4\) using calculus to compute
\(\sup_{t\geq 0}|\beta_2(t)|=|\beta_2(\frac{3+\sqrt{5}}{2})|\approx0.61801,\)
\begin{gather*}
    \sup_{t\geq 0}|\beta_3(t)|= \bigg|\beta_3\bigg(2+\sqrt{6}\cos \bigg(\frac{1}{3}\arctan \bigg(\frac{\sqrt{29}}{5} \bigg) \bigg)\bigg)\bigg|\approx 0.55191, \\
    \sup_{t\geq 0}|\beta_4(t)|=\bigg|\beta_4\bigg(3+2\sqrt{3}\cos\left(\frac{1}{3}\arctan\left(\sqrt{11}\right)\right)\bigg)\bigg|\approx 0.50755.
\end{gather*}
It follows that \(|\beta_j(t)|<2/e\) for all \(t\geq 0\) when \(j\geq 2\). Therefore \(j=1\), meaning that \(N=n\), \(\xi=1\), and \(a_n=|\beta_1(t)|=2te^{-t}\). The maximum of this function occurs at \(t=1\) where \(a_n=2/e\), with equality holding if and only if \(f=f_n\), as we wished to show. 
\end{proof}

\section{The Fej\'er lemma, positivity, and the proof of Theorem \ref{thm:2}}

Fix \(n\in\mathbb N\) and an extremal function \(\smash{f(z)=\sum_{j=0}^\infty a_jz^j}\) as in \eqref{extremal}, normalized so that \(a_0,a_n>0\). Recall that by \cite[Prop. 3(e)]{MartinSawyerTueroVukotic2015}, if \(P\) is as in Theorem \ref{thm:X} then \(\re P\geq 0\) on \(\overline{\mathbb D}\). Thus, the function \(T(\theta)=\re P(e^{i\theta})\) is a non-negative trigonometric polynomial of degree \(n\). To exploit this fact, we recall a well-known representation result.

\medskip

\begin{lem}[Fej\'er-Riesz]
If \(T\) is a non-negative trigonometric polynomial with degree \(n\), then there exists a polynomial \(p\) of degree \(n\) which is non-vanishing in \(\mathbb D\) such that \(T(\theta)=|p(e^{i\theta})|^2\). 
\end{lem}

\medskip

Thus, we can write \(\re P(e^{i\theta})=|p(e^{i\theta})|^2\) for a polynomial \(p\) of degree \(n\) with no zeros in \(\mathbb D\). Letting \(z_1,\dots, z_n\in \mathbb D^c\) denote the zeros of \(p\), we may write
\[
    p(z)=c\prod_{j=1}^n(z-z_j).
\]
As \(p\) is unique up to multiplication by a unimodular constant, we may also assume that \(c\) is real and positive. Since \(\re P(\alpha_j)=0\) for \(\alpha_1,\dots,\alpha_N\), we see that \(p(\alpha_j)=0\) for \(1\leq j\leq N\). More generally, any zero of \(\re P\) on \(\partial \mathbb D\) is a zero of \(p\). Now we are ready to prove Theorem \ref{thm:2}.

\begin{proof}[Proof of Theorem \ref{thm:2}]
First let \(e_k\) denote the \(k\)th elementary symmetric polynomial in \(n\) variables,
\[
    e_k(w_1,\dots, w_n)=\sum_{1\leq j_1<\cdots<j_k\leq n}\prod_{\ell=1}^kw_{j_\ell}.
\]
For convenience we also set \(e_0\equiv1\). Using these polynomials, we determine the coefficients of \(p\) in terms of its roots using Vieta's formulas,
\[
    p(z)=c\prod_{j=1}^n(z-z_j)=c\sum_{j=0}^n (-1)^{n-j}e_{n-j}(z_1,\dots,z_n)z^j.
\]
Our goal is to write the coefficients of \(\re P(e^{i\theta})\) in terms of \(z_1,\dots,z_n\). For simplicity, we abbreviate \(v_k=e_k(z_1,\dots,z_n)\) for \(1\leq k\leq n\) and set \(v_k=0\) for \(k\leq 0\) and \(k>n\). Combining the formula above with the identity \(\re P(e^{i\theta})=p(e^{i\theta})\overline{p(e^{i\theta})}\), we compute
\begin{align*}
    \re P(e^{i\theta})&=c^2\bigg(\sum_{j=0}^n (-1)^{n-j}v_{n-j}e^{ij\theta}\bigg)\bigg(\sum_{j=0}^n (-1)^{n-j}\overline{v_{n-j}}e^{-ij\theta}\bigg)\\
    &=(-1)^nc^2\sum_{j=0}^{2n}(-1)^{j}\bigg(\sum_{k=0}^jv_{n-j+k}\overline{v_k}\bigg)e^{i(j-n)\theta}.
\end{align*}
Using that \(v_k=0\) for \(k\not\in\{1,\dots, n\}\) and re-indexing, this simplifies to
\begin{align*}
    \re P(e^{i\theta})&=(-1)^nc^2\sum_{j=0}^{2n}(-1)^{j}\bigg(\sum_{k=-\infty}^\infty v_{n-j+k}\overline{v_k}\bigg)e^{i(j-n)\theta}=c^2\sum_{j=-n}^{n}(-1)^{j}\bigg(\sum_{k=-\infty}^\infty v_{k-j}\overline{v_k}\bigg)e^{ij\theta}.
\end{align*}
Writing \(d_j=(-1)^{j}\sum_{k=-\infty}^\infty v_{k-j}\overline{v_k}\) and accounting for vanishing terms, we find for \(j\geq 0\) that
\[
    d_j=(-1)^{j}\sum_{k=j}^n v_{k-j}\overline{v_k}=(-1)^{j}\sum_{k=0}^{n-j} v_{k}\overline{v_{j+k}}.
\]
It is easily verified that \(d_j=\overline{d_{-j}}\). To further refine this formula, we point out that since \(z_1,\dots,z_n\in\partial \mathbb D\) it holds that
\[
    \overline{v_{j+k}}=\overline{e_{j+k}(z_1,\dots, z_n)}=e_{j+k}\bigg(\frac{1}{z_1},\dots,\frac{1}{z_n}\bigg)=\frac{e_{n-j-k}(z_1,\dots,z_n)}{e_n(z_1,\dots,z_n)}.
\]
Altogether we see that \(\re P(e^{i\theta})=c^2\sum_{j=-n}^n d_je^{ij\theta}\), where \(d_{-j}=\overline{d_j}\) for \(j<0\), and for \(0\leq j\leq n\),
\[
    d_j=\frac{(-1)^{j}}{e_n(z_1,\dots,z_n)}\sum_{k=0}^{n-j} e_k(z_1,\dots,z_n)e_{n-j-k}(z_1,\dots,z_n).
\]

On the other hand, using that \(2\re z=z+\overline{z}\) we compute directly
\[
    \re P(e^{i\theta})=a_n+2\sum_{j=1}^n\re\{a_{n-j}e^{ij\theta}\}=a_n+\sum_{j=1}^na_{n-j}e^{ij\theta}+\sum_{j=1}^n\overline{a_{n-j}}e^{-ij\theta}.
\]
Equating coefficients in our two representations of \(\re P(e^{i\theta})\), which is justified by orthogonality of the functions \(\{e^{ij\theta}\}_{j\in\mathbb Z}\) on the circle, we find that \(a_{n-j}=c^2d_j\) for \(0\leq j\leq n\). Taking \(j=n\) gives \(c^2=(-1)^na_0e_n(z_1,\dots,z_n)\), so
\[
    a_{n-j}=(-1)^{n-j}a_0\sum_{k=0}^{n-j} e_k(z_1,\dots,z_n)e_{n-j-k}(z_1,\dots,z_n),
\]
and upon re-indexing, 
\begin{equation}\label{Rosetta stone}
    a_{j}=(-1)^{j}a_0\sum_{k=0}^{j} e_{j-k}(z_1,\dots,z_n)e_{k}(z_1,\dots,z_n).
\end{equation}

It is useful to inspect some special cases. Taking \(j=0\) in \eqref{Rosetta stone} and using \(v_{n-k}=\overline{v_k}v_n\) gives
\[
    a_{n}=(-1)^{n}a_0\sum_{k=0}^{n} v_kv_{n-k}=(-1)^{n}v_na_0\sum_{k=0}^{n} |v_k|^2.
\]
Our normalizations \(a_n,a_0>0\) thus ensure that \((-1)^nv_n\) is real and positive. It is also a unimodular constant, so \((-1)^nv_n=1\) and \(z_1\cdots z_n=v_n=(-1)^n\). Therefore
\[
    a_{n}=a_0\sum_{k=0}^{n} |v_k|^2=a_0\sum_{k=0}^{n} |e_k(\alpha_1,\dots,\alpha_n)|^2=a_0\bigg(2+\sum_{k=0}^{n-1} |e_k(\alpha_1,\dots,\alpha_n)|^2\bigg),
\]
where we have used that \(e_0\equiv 1\) and \(|e_n|=1\). 

The identity \(z_1\cdots z_n=(-1)^n\) also gives \(e_{k}(\overline{z_1},\dots, \overline{z_n})=(-1)^ne_{n-k}(z_1,\dots,z_n)\) for all \(k\). Taking \(e_j=0\) for \(j\not\in \{0,\dots, n\}\) and expanding using Vieta's formulas, we find that
\begin{align*}
    a_0\prod_{j=1}^n(1-z_jz)^2&=a_0\bigg(\prod_{j=1}^n(z-\overline{z_j})\bigg)^2=a_0\bigg(\sum_{j=0}^n(-1)^{j}e_{n-j}(\overline{z_1},\dots, \overline{z_n})z^{j}\bigg)^2\\
    &=a_0\sum_{j=0}^{2n}(-1)^{j}\bigg(\sum_{k=0}^je_{n-(j-k)}(\overline{z_1},\dots, \overline{z_n})e_{n-k}(\overline{z_1},\dots, \overline{z_n})\bigg)z^j\\
    &=a_0\sum_{j=0}^{2n}(-1)^{j}\bigg(\sum_{k=0}^je_{j-k}(z_1,\dots,z_n)e_{k}(z_1,\dots,z_n)\bigg)z^j.
\end{align*}
Comparing to \eqref{Rosetta stone}, it is clear that \(a_0\prod_{j=1}^n(1-z_jz)^2=a_0+a_1z+\cdots+a_nz^n\mod z^{n+1}\). 
\end{proof}

Next, we observe that the first \(n+1\) Taylor coefficients of \(\log f\) only depend on the first \(n+1\) coefficients of \(f\), meaning that 
\[
    \log f(z)=b_0+2\sum_{j=1}^n\log(1-zz_j)\mod z^{n+1}.
\]
Taylor expanding the logarithms and comparing the coefficients of \(z,\dots,z^n\) coefficients yields Corollary \ref{cor:wow!} at once.

\section{Assorted technical results}

The results of this section are not needed anywhere above, but they highlight interesting facts which may be useful in
further investigations of the conjecture. We begin with an elementary formula for the function \(g=\log f\). 

\begin{lem}\label{lem:formula for g}
Fix \(n\in\mathbb N\), let \(f=e^g\) be extremal for \(M_n\), and write \(f(0)=e^{-t}\). Then \(g\) has exactly \(N\) zeros \(\beta_1,\dots,\beta_N\in\partial\mathbb D\) and it holds that
\[
    g(z)=t\frac{\prod_{j=1}^N(z-\beta_j)}{\prod_{j=1}^N(z-\overline{\alpha_j})}.
\]
\end{lem}

\begin{proof}
Writing \(g=t\frac{h-1}{h+1}\), it is clear that \(\re g(z)=0\) only if \(|h(z)|=1\). As \(h\) is a finite Blaschke product, this only occurs if \(|z|=1\). Recalling that \(g(e^{i\theta})=i\varphi(\theta)\), we see that \(g(e^{i\theta})=0\) if and only if \(\varphi(\theta)=0\). On each interval \((\theta_j,\theta_{j+1})\), the function \(\varphi\) vanishes exactly once at some point \(\mu_j\), meaning that \(\varphi\) has \(N\) zeros. The points \(\beta_j=e^{i\mu_j}\) are thus the \(N\) zeros of \(g\). 

As \(g\) is rational, we may write \(g(z)=p(z)/q(z)\) for degree \(N\) polynomials \(p\) and \(q\). As \(q\) must vanish at \(\overline{\alpha_1,}\dots,\overline{\alpha_N}\) (where \(g\) has poles), and \(p\) must vanish at \(\beta_1,\dots,\beta_N\), we may write
\[
    g(z)=\gamma\frac{\prod_{j=1}^N(z-\beta_j)}{\prod_{j=1}^N(z-\overline{\alpha_j})},\qquad \gamma\in\mathbb C.
\]
To determine the value of \(\gamma\), we write \(h(z)=\xi\prod_{j=1}^N\frac{z-z_j}{1-\overline{z_j}z}\) for \(z_1,\dots, z_N\in \mathbb D\) and observe that
\[
    \gamma\frac{\prod_{j=1}^N(z-\beta_j)}{\prod_{j=1}^N(z-\overline{\alpha_j})}=g(z)=t\frac{h(z)-1}{h(z)+1}=t\frac{\xi\prod_{j=1}^N\frac{z-z_j}{1-\overline{z_j}z}-1}{\xi\prod_{j=1}^N\frac{z-z_j}{1-\overline{z_j}z}+1}=t\frac{\xi\prod_{j=1}^N(z-z_j)-\prod_{j=1}^N(1-\overline{z_j}z)}{\xi\prod_{j=1}^N(z-z_j)+\prod_{j=1}^N(1-\overline{z_j}z)}.
\]
The leading coefficients of the polynomials in the numerator and denominator on the right are respectively \(\xi-(-1)^N\prod_{z=1}^N\overline{z_j}\) and \(\xi+(-1)^N\prod_{z=1}^N\overline{z_j}\). As \(h(0)=0\), there exists some \(j\) such that \(z_j=0\), so the leading coefficients are equal. It follows that \(\gamma=t\).
\end{proof}

As usual, we may restrict to \(z\in\partial \mathbb D\) to recover a formula for the function \(\varphi\). 

\begin{cor}
Let \(f\) be extremal for \(M_n\), and write \(f(e^{i\theta})\overset{a.e.}{=}e^{i\varphi(\theta)}\). If \(\mu_1,\dots,\mu_N\) denote the zeros of \(\varphi\) and \(\alpha_j=e^{-i\theta_j}\), then 
\[
    \varphi(\theta)=t\frac{\prod_{j=1}^N\sin(\frac{\theta-\mu_j}{2})}{\prod_{j=1}^N\sin(\frac{\theta-\theta_j}{2})}.
\]
\end{cor}

\begin{proof}
A straightforward calculation shows that
\[
    g(e^{i\theta})=t\frac{\prod_{j=1}^N(e^{i\theta}-e^{i\mu_j})}{\prod_{j=1}^N(e^{i\theta}-e^{i\theta_j})}=t\frac{\prod_{j=1}^Ne^{i\frac{\mu_j}{2}}(e^{i\frac{\theta-\mu_j}{2}}-e^{i\frac{\mu_j-\theta}{2}})}{\prod_{j=1}^Ne^{i\frac{\theta_j}{2}}(e^{i\frac{\theta-\theta_j}{2}}-e^{i\frac{\theta_j-\theta}{2}})}=te^{\frac{i}{2}\sum_{j=1}^N(\mu_j-\theta_j)}\frac{\prod_{j=1}^N\sin(\frac{\theta-\mu_j}{2})}{\prod_{j=1}^N\sin(\frac{\theta-\theta_j}{2})}.
\]
Assuming without loss of generality that \(\theta_1<\mu_1<\theta_2<\mu_2<\cdots\), we observe that
\[
    \sum_{j=1}^N(\mu_j-\theta_j)=\sum_{j=1}^N|(\theta_j,\mu_j)|=|\{\theta\in\mathbb T\;|\; \varphi(\theta)<0\}|\in(0,2\pi).
\]
Therefore \(\frac{1}{2}\sum_{j=1}^N(\mu_j-\theta_j)\in (0,\pi)\). On the other hand, \(g(e^{i\theta})=i\varphi(\theta)\), so \(e^{\frac{i}{2}\sum_{j=1}^N(\mu_j-\theta_j)}=\pm i\). It follows that \(|\{\theta\in\mathbb T\;|\; \varphi(\theta)<0\}|=\pi\) and the claimed formula follow at once. 
\end{proof}

The preceding argument also shows that \(|\{\theta\in\mathbb T\;|\; \varphi(\theta)>0\}|=\pi\).  From the formula above, we observe that if \(\re P(z)=0\) for \(z\in\partial\mathbb D\setminus \{\overline{\alpha_1},\dots,\overline{\alpha_N}\}\), then \(\im Q(z)=0\) as well.

Next, we give a rudimentary estimate for the value of extremal functions at zero. 

\medskip

\begin{lem}\label{lower bound a_0}
Let \(f(z)=\sum_{j=0}^\infty a_jz^j\) be extremal for \(M_n\) and assume that \(a_0>0\). Then \(a_0\geq e^{-2n}\).
\end{lem}

\noindent\textit{Remark}: This is very far from the expected \(a_0=1/e\), so this bound is not very useful in practice.

\begin{proof}
Let \(U(x,y)=-\re \log f(x+iy)\). This is a non-negative harmonic function on the unit disc in \(\mathbb R^2\), so for any \(R<1\) and \(z=x+iy\) with \(|z|<R\) we have by Harnack's inequality that
\[
    \frac{R-|z|}{R+|z|}\,U(0,0)\leq U(x,y)\leq \frac{R+|z|}{R-|z|}\,U(0,0).
\]
Sending \(R\to 1\) and writing this in terms of \(f\) we find that
\[
    \frac{1-|z|}{1+|z|}|\log a_0|\leq |\log |f(z)||\leq \frac{1+|z|}{1-|z|}|\log a_0|.
\]
Taking exponentials gives \(a_0^\frac{1+|z|}{1-|z|}\leq |f(z)|\leq a_0^{\frac{1-|z|}{1+|z|}}\), meaning that \(|f(z)|\leq a_0^\frac{1-r}{1+r}\) for \(|z|\leq r<1\).

Now for \(\lambda\geq 0\) define functions \(f_\lambda(z)=f((1-\lambda)z)a_0^{-\frac{\lambda}{2-\lambda}}\). The preceding bound ensures that \(|f_\lambda|\leq 1\) for \(z\in \mathbb D\) so each \(f_\lambda\in\mathcal B_0\), and \(f_0=f\) is extremal. Thus,
\[
    0\geq \frac{d}{d\lambda}\bigg|_{\lambda=0}M_n(f_\lambda(z))=M_n\bigg(\frac{\partial}{\partial\lambda}\bigg|_{\lambda=0}f_\lambda(z)\bigg)=M_n\bigg(-zf'(z)-\frac{1}{2}f(z)\log a_0\bigg).
\]
Observe that \(M_n\) is \(\mathbb R\)-linear and \(M_n(f)=a_n\) while \(M_n(zf'(z))=na_n\). Thus, 
\[
    na_n+\frac{1}{2}a_n\log a_0\geq 0.
\]
Dividing by \(a_n>0\) and taking an exponential, the result follows immediately. 
\end{proof}

It is shown in \cite{MartinSawyerTueroVukotic2015} that if the polynomial \(P\) vanishes only on the unit circle, then the conjecture is true. The zeros of \(P\) can be localized to an annular region, which is precisely the unit circle when \(a_n=2a_0\) and \(a_1=\cdots=a_{n-1}=0\).

\medskip

\begin{lem}
Let \(P\) be as in Theorem \ref{thm:X}. The zeros of \(P\) belong to the annulus \(\{1\leq |z|\leq r\}\), 
\[
    r=\frac{1}{a_0}\sqrt{\sum_{j=0}^{n-1}|a_j|^2+\frac{a_n^2}{4}}.
\]
\end{lem}

\begin{proof}
Let \(R(z)=P(z)/2a_0\) so that \(R\) is monic, and set \(G(z)=R(z)-z^n\). If \(|z|=r\) then \(|z^n|=r^n\) while by the Cauchy-Schwarz inequality,
\[
    |G(z)|^2=\bigg|\sum_{j=1}^{n-1}\frac{a_j}{a_0}z^{n-j}+\frac{a_n}{2a_0}\bigg|^2\leq \bigg(\sum_{j=1}^{n-1}\frac{|a_j|}{a_0}r^{n-j}+\frac{a_n}{2a_0}\bigg)\leq \bigg(\sum_{j=1}^{n-1}\frac{|a_j|^2}{a_0^2}+\frac{a_n^2}{4a_0^2}\bigg)\bigg(\sum_{j=0}^{n-1}r^{2j}\bigg).
\]
The right-hand side is exactly \(r^{2n}-1\), so if \(|z|=r\) then \(|G(z)|^2\leq r^{2n}-1<r^{2n}=|z^n|^2\). Thus by Rouch\'e's theorem both \(z^n\) and \(R=z^n+G\) have the same number of zeros in \(\{|z|\leq r\}\). As \(R\) has degree \(n\) all of its zeros thus belong to this disc, and the same is true of \(P=2a_0R\). 
\end{proof}

\medskip

The rough estimate \(\sum_{j=0}^n|a_n|^2\leq 1\) shows that if \(P(z)=0\) then \(|z|\leq 1/a_0\). If \(f=f_n\) then \(r=1\) as we expect. Our use of Rouche's theorem was quite elementary, and conceivably a smaller value of \(r\) can be found using better estimates and properties of extremal functions. Furthermore, extra information is available about the coefficients of functions in \(\mathcal{B}_0\).

\medskip

\begin{lem}
Let \(\sum_{j=0}^\infty a_jz^j\in\mathcal{B}_0\). Then the following estimates hold:
\begin{gather*}
    |a_1|^2+|a_2|^2\leq \frac{32}{e^{4}}\approx 0.586,\\
    |a_1|^2+|a_2|^2+|a_3|^2\leq \frac{27}{2e^3} \approx0.672,\\
    |a_1|^2+|a_2|^2+|a_3|^2+|a_4|^2\leq \frac{96(33-19\sqrt{3})e^{2\sqrt{3}}}{e^{6}}\approx 0.692.
\end{gather*}
Equality holds in each estimate for \(\exp(t\frac{z-1}{z+1})\) for some \(t>0\).
\end{lem}

\medskip

Thus, the `mass' of functions in \(\mathcal B_0\) cannot concentrate to the first few coefficients. The bounds above follow from the subordination estimate in \cite[proof of Thm. 4]{NewmanShapiro1962} and elementary optimization. Such sharp bounds cannot be obtained so easily for five or more coefficients, since this would require one to compute in closed form the roots of a quintic polynomial. 

Another nice result, known already to the community and proved by Mart\'in et al. in \cite[\S2.1]{MartinSawyerTueroVukotic2015}, is that if extremal functions are unique then the conjecture is true. A slightly stronger result can be proved without much difficulty.

\medskip

\begin{lem}
Fix \(n\in\mathbb N\), and let \(p\) denote the smallest prime factor of \(n\). If \(M_n\) has fewer than \(p\) extremal functions, then \(\sup_{f\in\mathcal B_0}M_n(f)= \frac{2}{e}\) and \(f_n\) is the unique extremal function. 
\end{lem}

\begin{proof}
Let \(f\) be extremal for \(M_n\), set \(\omega=e^{2\pi i/n}\), and for \(1\leq j\leq p\) define \(f_j(z)=f(\omega^jz)\). Each \(f_j\) is extremal, so there exist distinct \(j,k\in\{1,\dots, n\}\) such that \(f_j=f_k\). Assume that \(j<k\) and set \(m=k-j\). It follows that \(1\leq m<p \) and that \(f(\omega^{m}z)=f_k(\omega^{-j}z)=f_j(\omega^{-j}z)=f(z)\). 

If \(\gcd(m,n)>1\) then it has a prime divisor \(q\), and \(q\leq m<p\), contradicting the fact that \(p\) is the smallest prime divisor of \(n\). Therefore \(\gcd(m,n)=1\) and \(m\) generates \(\mathbb Z/n\mathbb Z\), so \(\omega^m\) generates the \(n\)th roots of unity. Thus, \(\omega^{mk}=\omega\) for some \(k\in\mathbb N\), so \(f(\omega z)=f(z)\) and
\[
    f(z)=\frac{1}{n}\sum_{j=1}^nf(\omega^jz)=\sum_{j=0}^\infty a_{nj}z^{nj}. 
\]
This shows that \(f'\) vanishes to order \(n-1\) at zero. This proves the claim by Theorem \ref{thm:3} \textit{(3)}.
\end{proof}


Our final result is to highlight a nontrivial connection between the values of extremal functions at zero and the number of atoms \(N\) in the representation \eqref{extremal}. If we consider the problem of maximizing \(M_n\) on the restricted class \(\mathcal B_0(r)=\{f\in\mathcal B_0\;|\; |f(0)|\leq r\}\), then extremal functions need not be unique for all \(r\in (0,1)\).

\begin{lem}
Let \(r_0\approx 0.18047\) denote the smallest real solution to \(r\log r=-(2+\sqrt{5})e^{-\frac{3+\sqrt{5}}{2}}\), and set \(\mathcal B_0(r_0)=\{f\in\mathcal{B}_0\;|\;|f(0)|\leq r_0\}\). Then 
\[
    \sup_{f\in\mathcal B_0(r_0)}M_2(f)=(4+2\sqrt{5})e^{-\frac{3+\sqrt{5}}{2}},
\]
and equality holds if and only if
\[
    f(z)=\exp\bigg(\log\bigg(\frac{1}{r_0}\bigg)\cdot\frac{z^2-1}{z^2+1}\bigg)\quad\textrm{or}\quad f(z)=\exp\bigg(\frac{3+\sqrt{5}}{2}\cdot\frac{z-1}{z+1}\bigg).
\]
If \(r_0\) is replaced with any other number in \((0,1)\) then extremal functions are unique.
\end{lem}

This can be proved by a straightforward adaptation of the proof of the Krzy\.z conjecture in case \(n=2\) by Hummel et al. \cite{HummelScheinbergZalcman1977}, only optimizing over \(t\geq -\log r_0\) instead of over all \(t\geq 0\). This is left as an exercise. Extremal functions for the constrained problem therefore need not be unique, and can even have different numbers of atoms.


\bibliographystyle{siam}
\bibliography{references}
\address

\end{document}